\documentclass{amsart}
\usepackage[T1]{fontenc}
\usepackage{amsmath,amsthm,amssymb,amsfonts,enumerate,mathrsfs}
\newtheorem{theorem}{Theorem}
\newtheorem{lemma}{Lemma}
\newtheorem{proposition}{Proposition}
\newtheorem{corollary}{Corollary}

\theoremstyle{definition}
\newtheorem{definition}{Definition}
\newtheorem{conjecture}{Conjecture}
\newtheorem{remark}{Remark}

\newtheorem{question}{Question}

\newcommand{\R}{\mathbb{R}}
\newcommand{\N}{\mathbb{N}}
\newcommand{\Z}{\mathbb{Z}}

\title{On Some Properties of Accessible Sets}
\author{Oscar Quester}
\date{}

\begin{document}

\begin{abstract}

A set $D \subseteq \N$ is called $r$-large if every $r$-coloring of $\N$ admits arbitrarily long monochromatic arithmetic progressions $a,a+d,...,a+(k-1)d$ with gap $d \in D$. Closely related to largeness is accessibility; a set $D \subseteq \N$ is called $r$-accessible if every $r$-coloring of $\N$ admits arbitrarily long monochromatic sequences $x_1,x_2,...,x_k$ with $x_{i+1}-x_{i} \in D$. It is known that if $D \subseteq \N$ is $2$-large, then the gaps between elements in $D$ cannot grow exponentially. In this paper, we show that if $D$ is $2$-accessible, then the gaps between elements in $D$ cannot grow much faster than exponentially. Additionally, we show that the notion of accessibility is equivalent to that of topological recurrence. 

\end{abstract}
\maketitle

\section{Introduction}

An $r$-coloring of a set $A$ is a function $\chi : A \to [r]$, where $[r]=\{1,2,\dots,r\}$. We say a subset $A' \subseteq A$ is monochromatic under $\chi$ if $\chi $ is constant on $A'$. One of the goals of Ramsey theory is to find `order' in seemingly `random' structures. For example, van der Waerden's Theorem tells us that given any $r$-coloring of the positive integers, there will exist arbitrarily long monochromatic arithmetic progressions. The theorem places no requirement on the gap (or common difference), $d$, of the arithmetic progression -- it can be any natural number. With this in mind, in \cite{brown1999set}, the authors introduce the notion of largeness.

\begin{definition}

A set $D \subseteq \N=\{1,2,3,\dots\}$ is called \textit{$r$-large} if, for every $r$-coloring of $\N$, there exist arbitrarily long monochromatic arithmetic progressions with gap $d \in D$. If $D$ is $r$-large for every $r \in \N$, we say $D$ is \textit{large}. We call an arithmetic progression with gap $d \in D$ a \textit{$D$-AP}.
    
\end{definition}

In an attempt to learn more about largeness, in \cite{landman2010avoiding}, the authors study some Ramsey properties of a larger family of sequences.

\begin{definition}

Let $D \subseteq \N$. A $k$-term sequence of integers $x_1,x_2,\dots,x_k$ satisfying $x_{i+1}-x_{i} \in D$ for all $1 \leq i \leq k-1$ is called a $k$-term \textit{$D$-diffsequence}.    

\end{definition}

\begin{definition}
    
A set $D \subseteq \N$ is called \textit{$r$-accessible} if, for every $r$-coloring of $\N$, there exist arbitrarily long monochromatic sequences $x_1,x_2,\dots,x_k$ with $x_{i+1}-x_{i} \in D$; that is, there exist monochromatic $k$-term $D$-diffsequences for every $k \geq 1$. If $D$ is $r$-accessible for every $r \in \N$, we say $D$ is \textit{accessible}. 
    
\end{definition}

Equivalently -- by a compactness argument -- a set $D \subseteq \N$ is $r$-accessible if, for every $k \geq 1 $, there exists a least positive integer $\Delta = \Delta(D,k;r)$ such that whenever $[\Delta]$ is $r$-colored, there exists a monochromatic $k$-term $D$-diffsequence. The same is true for $r$-largeness. 
Given a set $D \subseteq \N$, we define the \textit{degree of accessibility} of $D$, $\operatorname{doa}(D)$, to be the greatest positive integer $r$ for which $D$ is $r$-accessible. If $D$ is $r$-accessible for every $r \in \N$, we write $\operatorname{doa}(D)=\infty$. It follows immediately from the definitions of largeness and accessibility that any set $D \subseteq \N$ that is $r$-large is automatically $r$-accessible. Thus, all large sets are accessible; though, as shown in \cite{jungic2005conjecture}, the converse need not be true. In Section 5, we explore the connections between accessibility and topological recurrence; namely, we show that accessible sets are precisely sets of topological recurrence. This connection motivates some new questions and helps to shed light on many interesting examples of sets that are accessible (but not large). 

It is a consequence of the polynomial van der Waerden Theorem \cite{bergelson1996polynomial} that given any polynomial $p(x) \in \mathbb{Q}[x]$ with positive leading term and zero constant term, the set $\{p(n) : n \in \N\} \cap \N$ is large (and thus accessible). This gives us a nice family of sets that are large. For example, the result tells us that the set $\{n^k : n \in \N\}$ is large for every $k \geq 2$. Note that the preceding set has natural density $0$ and is large, while the set $\{2n+1 : n \geq 0\}$ of odd positive integers has natural density $1/2$ and is not even $2$-accessible (color the odd positive integers red and the even positive integers blue). This suggests that the natural density of a set does not give us much information in regards to largeness or accessibility. However, in \cite{brown1999set}, the authors show that if $D=\{d_1,d_2,d_3,\dots\}$ is a set of positive integers satisfying $d_{n+1} \geq 3d_{n}$ for all $n \geq 1$, then $D$ is not $2$-large. In Question 5.5 of \cite{clifton2024new}, the author asks if a similar result holds for $2$-accessibility: \textit{Does there exist an absolute constant $C \geq 2$ such that there is no $2$-accessible set $D = \{d_1,d_2,d_3,\dots\}$ satisfying $d_{n+1} > Cd_n$ for all $n \geq 1$? If yes, what is the smallest such $C$?} We prove the following theorem.

\begin{theorem} \label{theorem :raccgrowth}

Let $r \geq 2$ and let $D=\{d_1<d_2<d_3<\dotsb \}$ be a set of positive integers. If there exist a real number $\delta>0$ and a positive integer $N$ such that $$d_{n+1} \geq (2+(r-1)^{-1}+\delta)d_{n}$$ for all $n \geq N$, then $\operatorname{doa}(D) \leq r-1$.

\end{theorem} 

Letting $r=2$, we immediately obtain the following corollary.

\begin{corollary} \label{cor:2accgrowth}

Let $D=\{d_1<d_2<d_3<\dotsb \}$ be a set of positive integers. If there exist a real number $\delta>0$ and a positive integer $N$ such that $$d_{n+1} \geq (3+\delta)d_{n}$$ for all $n \geq N$, then $D$ is not $2$-accessible.
    
\end{corollary}

\begin{remark} 

We note that since the upper bound $\operatorname{doa}(D) \leq r-1$ of Theorem \ref{theorem :raccgrowth} grows with $r$, the result is only interesting for $r=2$ and other small values of $r$. In general, if $d_{n+1} \geq (1+\delta)d_n$, then $\operatorname{doa}(D) < \infty$. In particular, a simple argument shows that if $d_{n+1} \geq (2+\delta)d_n$, then $\operatorname{doa}(D) \leq 15$. See \cite{katznelson2001chromatic,peres2010two,ruzsa2002distance}, which study the chromatic number $\chi(\Z_D)$ of graphs $\Z_D=(\Z,E)$, where $(x,y) \in E$ if and only if $|x-y| \in D.$ We can similarly define $\N_D=(\N,E)$, where $(x,y) \in E$ if and only if $|x-y| \in D$. It follows immediately that if $\chi(\N_D) \leq r$ then $\operatorname{doa}(D) \leq r-1$. However, no growth rate condition is sufficient to conclude $\chi(\N_D) \leq 2$ (and thus $\operatorname{doa}(D)=1$). Indeed, take $D=\{(2k)^n : n \geq 0\}$. Then, $\{1,2k+1,4k+1,\dots,(2k)^2+1\}$ is an odd cycle in $\N_D$, and thus $\chi(\N_D) \geq 3$. 
 
\end{remark}

Corollary \ref{cor:2accgrowth} is strikingly similar to the above mentioned growth rate condition for $2$-largeness, the only difference being the added $\delta>0$. The appearance of the constant $3$ in both results seems to be a coincidence, however. As we will see in Section 2, the condition $d_{n+1} \geq 3d_n$ in the $2$-largeness result can be weakened to $d_{n+1} \geq (1+\delta)d_n$, where $\delta >0$ is any positive constant. On the other hand, the set $\{2^n : n \geq 0\}$ is $2$-accessible and satisfies $d_{n+1} \geq 2d_{n}$.

In \cite{ardal2008ramsey}, the authors ask about the degree of accessibility of the set of (positive) Fibonacci numbers $F=\{1,2,3,5,\dots\}$ and show that $2 \leq \operatorname{doa}(F) \leq 5$. Recently, in \cite{wesley2022improved}, the author lowers the upper bound, showing $\operatorname{doa}(F) \leq 3$. The author's proof of this improved upper bound uses various properties of the Fibonacci sequence to construct a $4$-coloring of $\N$ that avoids arbitrarily long monochromatic $F$-diffsequences. In this paper, we arrive at the same upper bound of $\operatorname{doa}(F) \leq 3$ using mainly the growth rate of the Fibonacci sequence.

\section{Some Background and Basic Results}

First, we note that if $D \subseteq \N$ is a finite set, then $D$ is neither $2$-large nor $2$-accessible. Indeed, letting $m=\max(D)$, the $2$-coloring $\chi : \N \to \{1,2\}$ defined by $$\chi (x)= \begin{cases}
        1 :  x \bmod{2m} \in \{1,2,\dots,m\} \\
        2 : \text{otherwise } 
    \end{cases}$$

\noindent admits no monochromatic $(m+1)$-term $D$-AP or $(m+1)$-term $D$-diffsequence. 

The following result gives us a quite powerful necessary condition for a set $D$ to be $2$-large.

\begin{theorem}[\cite{brown1999set}, Theorem 2.1]
\label{theorem :2largemultiples}

If $D \subseteq \N$ is $2$-large then $D$ contains a multiple of $m$ for every $m \in \N$. 
    
\end{theorem}

The next result tells us that we can remove a finite number of elements from an $r$-large set and preserve $r$-largeness. This is not in general true for accessibility; for example, the set of odd positive integers $D=\{1,3,5,7,\dots\}$ is not $2$-accessible but $D\cup\{2\}$ is $3$-accessible (\cite{landman2014ramsey}, Theorem 10.27).

\begin{theorem}[\cite{brown1999set}, Lemma 1]
\label{theorem :rlargefinite}

Let $D \subseteq \N$ be $r$-large. If $F$ is a finite set, then $D \setminus F$ is $r$-large.
    
\end{theorem}

The following tells us that if $D=D_1 \cup D_2$ is large, then at least one of $D_1$ or $D_2$ is also large. By induction, the result is easily extended to any finite union.

\begin{theorem}[\cite{brown1999set}, Theorem 2.4] 
\label{theorem :largeunion}

If $D=D_1 \cup D_2$ is large, then either $D_1$ or $D_2$ is large.
    
\end{theorem}

We also have the previously mentioned growth rate condition for $2$-largeness.

\begin{theorem}[\cite{brown1999set}, Theorem 2.2]
\label{theorem :2largegrowthweak}
Let $D=\{d_1,d_2,d_3,\dots\}$ be a set of positive integers such that $d_{n+1} \geq 3d_{n}$ for all $n \geq 1$. Then $D$ is not $2$-large.
    
\end{theorem}

As previously mentioned, the condition of Theorem \ref{theorem :2largegrowthweak} can be weakened. To do so, we introduce some new notation. A sequence of positive integers $\mathcal{S}=(m_n)_{n=1}^{\infty}$ is said to be \textit{lacunary} if there exists $\lambda>1$ such that $m_{n+1} \geq \lambda m_n$ for all $n \geq 1$ (if the elements of $D \subseteq \N$ form a lacunary sequence we will also call $D$ lacunary). Let $\{x\}=x-\lfloor x \rfloor$ denote the fractional part of $x \in \R$. Let $\|x\|=\min (\{x\},1-\{x\})$, that is, the distance from $x \in \R$ to the nearest integer. It has long been known that if $\mathcal{S}=(m_n)_{n=1}^{\infty}$ is lacunary then there exist $\alpha \in \R$ and $\varepsilon>0$ such that $\inf_{n \in \N}\|\alpha m_n\|=\varepsilon$ (this was first proven in 1926 by Khintchine \cite{khintchine1926klasse}, forgotten, and then subsequently proven by both Pollington and Mathan in the late 1970s \cite{de1980numbers,pollington1979density}). More recently, Peres and Schlag proved the following.

\begin{theorem}[\cite{peres2010two}, Theorem 3.1] 
\label{theorem :lacunary}

Let $\mathcal{S}_1,\mathcal{S}_2,\dots,\mathcal{S}_\ell$ be lacunary sequences and let $\mathcal{S}=\mathcal{S}_1 \cup \mathcal{S}_2 \cup \dotsb \cup \mathcal{S}_\ell.$ Then, there exist $\alpha \in \R$ and $\varepsilon>0$ such that $$\inf_{s \in\mathcal{S}}\|\alpha s\|=\varepsilon.$$
    
\end{theorem}

The next result was first proven in \cite{host2016variations} (Lemma 7.4) and the following quantitative version of it appears in \cite{farhangi2021distance}.

\begin{theorem}[\cite{farhangi2021distance}, Theorem 2]
\label{theorem :2largelonley}

Let $D \subseteq \N$ be a set of positive integers such that there exist $\alpha \in \R$ and $\varepsilon>0$ satisfying $\|\alpha d\| \geq \varepsilon$ for all $d \in D$. Then, there exists a $2$-coloring of $\N$ that admits no monochromatic $D$-AP of length $\ell=\lceil 1/2\varepsilon \rceil +1$.
    
\end{theorem}

\begin{remark}

Theorem \ref{theorem :2largelonley} implies Theorem \ref{theorem :2largemultiples} since if $D \subseteq \N$ contains no multiple of $m \in \N$, then $\|m^{-1}d\| \geq m^{-1}$ for all $d \in D$. 
    
\end{remark}

As a consequence of Theorems \ref{theorem :lacunary} and \ref{theorem :2largelonley}, we obtain the following significant improvement to Theorem \ref{theorem :2largegrowthweak}.

\begin{theorem}[\cite{farhangi2021distance}]
\label{theorem :2largegrowthstrong}

Let $D_1,D_2,\dots,D_n \subseteq \N$ be lacunary sets and let $D=\bigcup_{i=1}^{n}D_i$. Then, $D$ is not $2$-large. 
    
\end{theorem}

The following is a corollary of the polynomial van der Waerden Theorem and gives us a nice family of large sets.

\begin{theorem}[\cite{brown1999set}, Theorem 3.1]
\label{theorem :rangeofpolynomiallarge}

Let $p(x)$ be a polynomial with integer coefficients, a positive leading term, and $p(0)=0$. Then, $\N \cap \{p(n) : n \in \N\}$ is large.

\end{theorem}

\begin{remark}

In the above theorem, it states that the polynomial must have integer coefficients. This is not necessary; the polynomial need only have rational coefficients. To see why, let $p(x)=a_nx^n+a_{n-1}x^{n-1}+\dotsb+a_1x$ be a polynomial with $a_n,a_{n-1},\dots,a_1 \in \mathbb{Q}$. Then, we can write $$a_n=\frac{p_n}{q_n},a_{n-1}=\frac{p_{n-1}}{q_{n-1}},\dots,a_1=\frac{p_1}{q_1}$$ where $\gcd(p_i,q_i)=1$. Let $\ell = \operatorname{lcm}(q_1,q_2,\dots,q_n).$ Then, $q(x)=p(\ell x)$ is a polynomial with integer coefficients and we can apply Theorem \ref{theorem :rangeofpolynomiallarge} to $q(x)$.
    
\end{remark}

Reguarding $r$-accessibility, the following gives us a sufficient condition.

\begin{theorem}[\cite{landman2007avoiding}, Lemma 2.1]
\label{theorem :accessiblecriteria1}

Let $c \geq 0$ and $r \geq 2$, and let $D \subseteq \N$. If every $(r-1)$-coloring of $D$ yields arbitrarily long monochromatic $(D+c)$-diffsequences, then $D+c$ is $r$-accessible.
    
\end{theorem}

Most useful is the case when $r=2$ and $c=0$, in which case we have the following.

\begin{corollary} \label{theorem :2acccon}

Let $D \subseteq \N$. If $D$ itself contains arbitrarily long $D$-diffsequences, then $D$ is $2$-accessible.
    
\end{corollary}

This corollary gives us the lower bound of $2$ for the degree of accessibility of the Fibonacci numbers, $F$. Since $F=\{f_n : n \in \N\}$ and $f_n=f_{n-1}+f_{n-2}$, we have $f_{n}-f_{n-1}=f_{n-2}$ for all $n \geq 3$. Thus, the set of Fibonacci numbers contains arbitrarily long Fibonacci diffsequences, and is thus $2$-accessible by Corollary \ref{theorem :2acccon}.

Given some set $D \subseteq \N$, it is often useful to work with a subset of $D$, as opposed to $D$ itself. The next theorem is a generalization of a result that appears in \cite{landmanramsey} and \cite{wesley2022improved}.

\begin{theorem}
\label{theorem :accessiblecriteria2}
Let $A \subseteq \N$ and let $d \in \N$. Let $B=\left\{\frac{a}{d}: a \in A \right\} \cap \N$ and let $C=\{a \in A: a \equiv 0\pmod{d}\}$. Then the following hold:

\begin{enumerate}[$(i)$]

\item if $A$ is $rd$-accessible, then $B$ is $r$-accessible;

\item if $A$ is $rd$-accessible, then $C$ is $r$-accessible.

\end{enumerate}

\begin{proof}

In both cases, we will prove the contrapositive. First, assume that $B$ is not $r$-accessible. Then, there exists a positive integer $k \geq 1$ and an $r$-coloring $\chi : \N \to [r]$ of $\N$ such that $\chi$ admits no monochromatic $k$-term $B$-diffsequence. Assume, for the sake of contradiction, that $A$ is $rd$-accessible. Define the $rd$-coloring $\chi' : \N \to \{0,1,\dots,d-1\} \times [r]$ by $$\chi' (x)=\left(x\bmod{d},\chi\left(\frac{x-x\bmod{d}}{d}+1\right)\right).$$ Since $A$ is $rd$-accessible, there exists a monochromatic $k$-term $A$-diffsequence, say $X=\{x_1,x_2,\dots,x_k\}$, under $\chi'$. By definition of $\chi'$, every element of $X$ is congruent to $c\pmod{d}$ for some $c\in \{0,1,\dots,d-1\}$. Thus, $$\chi\left(\frac{x_1-c}{d}+1\right)=\chi\left(\frac{x_2-c}{d}+1\right)= \dotsb =\chi\left(\frac{x_k-c}{d}+1\right)$$ and $$\left(\frac{x_{i+1}-c}{d}+1\right)-\left(\frac{x_i-c}{d}+1\right)=\left(\frac{x_{i+1}-x_i}{d}\right) \in B$$ for all $1 \leq i \leq k-1$, a contradiction.

Now, assume that $C$ is not $r$-accessible. Then, there exists a positive integer $k \geq 1$ and an $r$-coloring $\chi : \N \to [r]$ of $\N$ such that $\chi$ admits no monochromatic $k$-term $C$-diffsequence. Assume, for the sake of contradiction, that $A$ is $rd$-accessible. Define the $rd$-coloring $\chi' : \N \to \{0,1,\dots ,d-1\} \times [r]$ by $$\chi' (x)=\left(x\bmod{d},\chi(x)\right).$$ Since $A$ is $rd$-accessible, there exists a monochromatic $k$-term $A$-diffsequence, say $X=\{x_1,x_2,\dots,x_k\}$, under $\chi'$. By definition of $\chi'$, every element of $X$ is congruent to $c\pmod{d}$ for some $c\in \{0,1,\dots,d-1\}$. Thus, we have $$\chi(x_1)=\chi(x_2)= \dotsb =\chi(x_k)$$ and $x_{i+1}-x_i \equiv 0\pmod{d}$, so $x_{i+1}-x_i \in C$ for all $1 \leq i \leq k-1$, which contradicts our assumption. \end{proof}

\end{theorem}

\begin{remark}

Theorem \ref{theorem :accessiblecriteria2} will be useful for showing $\operatorname{doa}(F) \leq 3$. By letting $F_E=\{f_{3n} : n \in \N\}=\{2,8,34,\dots\}$, we see that $F_E=\{f \in F : f \equiv 0 \pmod{2}\}$. Thus, if we can show $\operatorname{doa}(F_E)=1$, we will have $\operatorname{doa}(F) \leq 3$.
    
\end{remark}

\section{Proof of Theorem 1}

We begin by giving a proof of our main result, Theorem \ref{theorem :raccgrowth}. Our proof, and result, is similar in flavor to that of Theorem 2 in \cite{ruzsa2002distance}, and the same sort of nested interval construction can be found in \cite{dubickas2006fractional}. The first part of the proof uses the fact that $d_{n+1} \geq (2+(r-1)^{-1}+\delta)d_n$ to construct an $\alpha \in \R$ satisfying $\{\alpha d_n \} \in \left[\varepsilon,\frac{r-1}{r}\right]$ for all $n \in \N$, where $\varepsilon>0$ is some fixed positive constant. The second part shows that this condition on the fractional parts $\{\alpha d_n \}$ is sufficient to conclude $D$ is not $r$-accessible.

\begin{proof}[Proof of Theorem 1]

Let $D=\{d_1<d_2<d_3<\dotsb \}$ be a set of positive integers. Suppose that there exist $\delta>0$ and a positive integer $N$ such that $d_{n+1} \geq (2+(r-1)^{-1}+\delta)d_{n}$ for all $n \geq N$. Set $q_n=d_{n+N-1}$. Then, $q_{n+1} \geq (2+(r-1)^{-1}+\delta)q_n$ for all $n \geq 1$. We construct a sequence of nested intervals $I_1 \supseteq I_2 \supseteq I_3 \supseteq \dotsb$ such that $$I_n = \left[\frac{z_n+\varepsilon}{q_n},\frac{rz_n+(r-1)}{rq_n}\right],$$ where $z_n \in \N \cup \{0\}$ and $\varepsilon > 0$ is fixed. Set $$\varepsilon=\frac{\delta(r-1)}{r\left(2+(r-1)^{-1}+\delta \right)}$$ and define $I_1=\left[\frac{\varepsilon}{q_1},\frac{r-1}{rq_1}\right].$ Assume we have defined $I_1 \supseteq I_2 \supseteq \dotsb \supseteq I_k$ as above. Thus, we have $$I_k=\left[\frac{z_k+\varepsilon}{q_k},\frac{rz_k+(r-1)}{rq_k}\right].$$ Note that $I_k=I^{(1)}_k \cup I^{(2)}_k$ where $$I^{(1)}_k=\left[\frac{z_k+\varepsilon}{q_k},\frac{z_k+\varepsilon}{q_k}+\frac{1}{q_{k+1}}\right]$$ and $$I^{(2)}_k=\left[\frac{z_k+\varepsilon}{q_k}+\frac{1}{q_{k+1}},\frac{rz_k+(r-1)}{rq_k}\right].$$ Since $|I^{(1)}_k| = \frac{1}{q_{k+1}}$, there must exist $z_{k+1}\in \N$ such that $\frac{z_{k+1}}{q_{k+1}}\in I^{(1)}_k$. We claim that $$I_{k+1}=\left[\frac{z_{k+1}+\varepsilon}{q_{k+1}},\frac{rz_{k+1}+(r-1)}{rq_{k+1}}\right] \subseteq I_k.$$ 

\noindent Since $\frac{z_{k+1}}{q_{k+1}}\in I_k^{(1)}$ and $I_k=I_k^{(1)} \cup I_k^{(2)}$, the above inclusion will hold if $|I^{(2)}_k| \geq \frac{r-1}{rq_{k+1}}.$ By construction, $$|I^{(2)}_k|=\frac{(r-1)-r\varepsilon}{rq_k}-\frac{1}{q_{k+1}}.$$ Using the fact that $q_{k+1} \geq (2+(r-1)^{-1}+\delta)q_{k}$, we obtain
\begin{align*}
    \frac{(r-1)-r\varepsilon}{rq_k}-\frac{1}{q_{k+1}} & \geq \frac{\left((r-1)-r\varepsilon\right)\left(2+(r-1)^{-1}+\delta\right)}{rq_{k+1}} - \frac{1}{q_{k+1}} \\
    & = \frac{2(r-1)-2r\varepsilon + 1 -r\varepsilon(r-1)^{-1}+\delta(r-1)-r\varepsilon \delta - r}{rq_{k+1}} \\
    & = \frac{r-1}{rq_{k+1}}+\frac{\delta(r-1) - 2r\varepsilon - r\varepsilon(r-1)^{-1}-r\varepsilon \delta}{rq_{k+1}} \\ 
    & = \frac{r-1}{rq_{k+1}}+\frac{\delta(r-1)-\varepsilon r(2+(r-1)^{-1}+\delta)}{rq_{k+1}} \\
    & = \frac{r-1}{rq_{k+1}},
\end{align*} where the last equality follows from the definition of $\varepsilon.$ Now, since the $I_n$ are nested and $|I_n| \to 0$ as $n \to \infty$, we have $\bigcap_{n=1}^{\infty}I_n=\{\alpha\}$. By construction, this $\alpha \in \R$ satisfies $\{\alpha q_n\} \in \left[\varepsilon,\frac{r-1}{r}\right]$ for all $n \in \N$. Moreover, since $I_1=\left[\frac{\varepsilon}{q_1},\frac{r-1}{rq_1}\right]$, we have $\frac{\varepsilon}{d_N} = \frac{\varepsilon}{q_1} \leq \alpha \leq \frac{r-1}{rq_1}=\frac{r-1}{rd_N}$. Thus, letting $\varepsilon_1=\varepsilon(d_1/d_N)$, we have $\{\alpha d_n\} \in \left[\varepsilon_1,\frac{r-1}{r}\right]$ for all $n \in \N$.

Now, consider the $r$-coloring $\chi : \N \to [r]$ defined by $$\chi(x)=\begin{cases}
        1:\{\alpha x \} \in \left[0,\frac{1}{r}\right) \\
        2: \{\alpha x \} \in \left[\frac{1}{r},\frac{2}{r}\right) \\
        \vdots& \\
        r: \{\alpha x \} \in \left[\frac{r-1}{r},1\right).
    \end{cases}$$
We claim this coloring avoids monochromatic $\left(\lceil (r\varepsilon_1)^{-1}\rceil +1 \right)$-term $D$-diffsequences. To see this, let $k=\left(\lceil (r\varepsilon_1)^{-1}\rceil +1 \right)$ and assume, for the sake of contradiction, that $X=\{x_1,x_2,\dots,x_k\}$ is a monochromatic $k$-term $D$-diffsequence, say of color $c$. Then, $\{\alpha x_i \} \in \left[\frac{c-1}{r},\frac{c}{r}\right)$ for all $1 \leq i \leq k$. We claim $$\{\alpha x_{i+1}\}=\{\alpha(x_{i+1}-x_i)\}+\{\alpha x_i\}$$ for all $1 \leq i \leq k-1$. Indeed, assume, for the sake of contradiction, that $\{\alpha x_{i+1}\}=\{\alpha(x_{i+1}-x_i)\}+\{\alpha x_i\}-1$. Since $x_{i+1}-x_{i} \in D$ for all $1 \leq i \leq k-1$, we have $\varepsilon_1 \leq \{\alpha (x_{i+1}-x_i)\} \leq \frac{r-1}{r}$ for all $1 \leq i \leq k-1$. Thus, we have $$\{\alpha x_{i+1}\}=\{\alpha(x_{i+1}-x_i)\}+\{\alpha x_i\}-1 < \frac{r-1}{r}+\frac{c}{r}-1 = \frac{c-1}{r},$$ a contradiction with $\chi(x_{i+1})=c$. Iterating this fact, we obtain $$\{\alpha x_k\} = \{\alpha(x_{k}-x_{k-1})\}+\dotsb+\{\alpha(x_2-x_1)\}+\{\alpha x_1\} \geq (k-1)\varepsilon_1 + \frac{c-1}{r} \geq \frac{c}{r}.$$ Thus, $\chi(x_k) \not = c$, a contradiction. It follows that $\operatorname{doa}(D) \leq r-1.$ \end{proof}

\begin{remark}

The coloring in the proof of Theorem \ref{theorem :raccgrowth} can equivalently be defined as $$\chi(x)=r+\sum_{i=1}^{r-1}\left\lfloor \{\alpha x\} - \frac{i}{r}\right\rfloor=1 + \sum_{i=1}^{r-1}\left\lfloor \alpha x + \frac{i}{r}\right\rfloor-\lfloor \alpha x \rfloor.$$ In addition, during the course of our proof we obtained a new way to show a set is not $r$-accessible; indeed, we have the following. 

\begin{theorem}
\label{theorem :newraccessible}

If $D=\{d_1,d_2,d_3,\dots\}$ is a set of positive integers such that there exist $\alpha \in \R$ and $\varepsilon>0$ satisfying $\{\alpha d_n\} \in \left[\varepsilon,\frac{r-1}{r}\right]$ for all $n \in \N$, then there exists an $r$-coloring of $\N$ with no monochromatic $(\lceil (r \varepsilon)^{-1}\rceil+1)$-term $D$-diffsequence.
    
\end{theorem}
    
\end{remark}

Using the fact that the even Fibonacci numbers satisfy the recurrence $e_{n}=4e_{n-1}+e_{n-2}$ for all $n \geq 3$ (with $e_1=2$ and $e_2=8$), it follows by Corollary \ref{cor:2accgrowth} that $\operatorname{doa}(F_E)=1$. Thus, by the remark following the proof of Theorem \ref{theorem :accessiblecriteria2}, we have $\operatorname{doa}(F) \leq 3$.

\section{Colorings Avoiding Long Monochromatic $F$-APs and $F_E$-Diffsequences}

As observed in \cite{farhangi2021distance}, Theorem \ref{theorem :2largegrowthstrong} immediately implies that the set of Fibonacci numbers $F=\{1,2,3,5,\dots\}$ is not $2$-large. This relies only on the fact that the Fibonacci sequence is lacunary, and hence does not allow one to demonstrate an explicit $2$-coloring of $\N$ that avoids arbitrarily long monochromatic $F$-APs.

In \cite{wesley2022improved}, the author constructs a quite intricate $2$-coloring of $\N$ that avoids $5$-term monochromatic $F$-APs. The proof is rather technical, however. We give an elementary proof that there exists a $2$-coloring of $\N$ that avoids $6$-term monochromatic $F$-APs.

\begin{proposition}
\label{cor:fibAP}

There exists a $2$-coloring of $\N$ that avoids $6$-term monochromatic $F$-APs. Consequently, $F$ is not $2$-large.

\begin{proof}

Let $\phi=(1+\sqrt{5})/2$, $\bar{\phi}=(1-\sqrt{5})/2$, and $\alpha=1/\sqrt{5}.$ Then, $$f_n=\alpha \phi^n - \alpha \bar{\phi}^{n}$$ for all $n \geq 0$. We note that $$\phi^{-1}f_n=f_{n-1}-\bar{\phi}^n$$ and $$\phi f_n=f_{n+1}-\bar{\phi}^n.$$ Thus, $$\sqrt{5}f_n=(\phi^{-1}+\phi)f_n=f_{n-1}+f_{n+1}-2\bar{\phi}^n$$ for all $n \geq 1$. We will also use the fact that $f_{n-1}+f_{n+1} \not \equiv 0 \pmod{8}$ for all $n \geq 1$. This can be easily verified by checking the first few cases and then noting that $f_{n} \equiv f_{n+12} \pmod{8}$ (the Fibonacci sequence is periodic modulo $m$ for every $m$; when $m=8$, the period is $12$). Direct computation gives us $$\left\| \frac{\sqrt{5}f_n}{8}\right\| > 0.16$$ for $1 \leq n \leq 4$. Then, for $n \geq 5$, we note that $$-0.014<-\frac{1}{4} \left(\frac{1-\sqrt{5}}{2}\right)^{6} \leq -\frac{\bar{\phi}^n}{4} \leq -\frac{1}{4}  \left(\frac{1-\sqrt{5}}{2}\right)^{5}<0.023.$$ 

\noindent Combining this with the fact that $$\left\{\frac{f_{n-1}+f_{n+1}}{8}\right\} \in \left\{\frac{1}{8},\frac{2}{8},\dots,\frac{7}{8}\right\}$$ for all $n \geq 1$, we have $$\left\|\frac{\sqrt{5}f_n}{8}\right\| > 0.1$$ for all $n \geq 1$. Since $6=\left\lceil \frac{1}{2(0.1)}\right\rceil+1$, it follows by Theorem \ref{theorem :2largelonley} that there exists a $2$-coloring of $\N$ with no monochromatic $6$-term $F$-AP. Indeed, the $2$-coloring $\chi: \N \to \{1,2\}$ defined by $$\chi(x)=\begin{cases}
    1 : \left\{\frac{\sqrt5}{8} x\right\} \in \left[0,\frac{1}{2}\right) \\
    2 : \left\{\frac{\sqrt5}{8} x\right\} \in \left[\frac{1}{2},1\right) 
\end{cases}$$ avoids $6$-term monochromatic $F$-APs. \end{proof}
    
\end{proposition}

At the end of Section 3 we concluded that $\operatorname{doa}(F_E)=1$ using only the growth rate of the even Fibonacci numbers. Consequently, we did not demonstrate an explicit $2$-coloring of $\N$ that avoids arbitrarily long monochromatic $F_E$-diffsequences.

In \cite{wesley2022improved}, the author constructs a $4$-coloring of $\N$ that avoids $4$-term monochromatic $F$-diffsequences. We give a shorter and more elementary proof of the existence of such a coloring. We use Theorem \ref{theorem :newraccessible} to prove there exists a $2$-coloring of $\N$ that avoids $4$-term monochromatic $F_E$-diffsequences. This coloring then immediately extends (using Theorem \ref{theorem :accessiblecriteria2}) to a $4$-coloring of $\N$ that avoids $4$-term monochromatic $F$-diffsequences.

\begin{proposition}
    
\label{cor:fibdiff}

There exists a $2$-coloring of $\N$ that avoids $4$-term monochromatic $F_E$-diffsequences, where $F_E=\{2,8,34,\dots\}$ is the set of even Fibonacci numbers. Consequently, $F_E$ is not $2$-accessible.

\begin{proof}

As in Proposition \ref{lem:corr}, we let $\phi=(1+\sqrt{5})/2$, $\bar{\phi}=(1-\sqrt{5})/2$, and $\alpha=1/\sqrt{5}$ so that $$f_n=\alpha \phi^n - \alpha \bar{\phi}^n$$ for all $n \geq 0$. Next, note that $$\phi f_n=f_{n+1}-\bar{\phi}^n,$$ from which it follows that $$(1+\phi)f_n=f_{n}+f_{n+1}-\bar{\phi}^n$$ for all $n \geq 1.$ We claim that $f_{3n}+f_{3n+1} \equiv 1 \pmod{4}$ for all $n \geq 0$. This is true for $n=0$ and $n=1$. Now, assume that $k \geq 2$ and $f_{3m}+f_{3m+1} \equiv 1 \pmod{4}$ for all $m<k$. Then, using the fact that $f_{n+6} \equiv f_n \pmod{4}$, we have $f_{3k}+f_{3k+1} \equiv f_{3(k-2)}+f_{3(k-2)+1} \equiv 1 \pmod{4}$. Hence, it follows that $f_{3n}+f_{3n+1} \equiv 1\pmod{4}$ for all $n \geq 1$, that is, $$\left\{\frac{f_{3n}+f_{3n+1}}{4}\right\}=\frac{1}{4}$$ for all $n \geq 1$. Finally, we have $$-0.04 < -\frac{1}{4}  \left(\frac{1-\sqrt{5}}{2}\right)^{4} \leq -\frac{\bar{\phi}^{3n}}{4} \leq -\frac{1}{4}  \left(\frac{1-\sqrt{5}}{2}\right)^{3}<0.06$$ for all $n \geq 1$, and thus $$0.21<\left\{(1+\phi) \frac{f_{3n}}{4}\right\}<0.31$$ for all $n \geq 1$. Since $4=\left\lceil \frac{1}{2(0.21)} \right\rceil+1$, it follows by Theorem \ref{theorem :newraccessible} that there exists a $2$-coloring of $\N$ with no monochromatic $4$-term $F_E$-diffsequence. Indeed, the $2$-coloring $\chi : \N \to \{1,2\}$ defined by $$\chi(x)=\begin{cases}
    1 : \left\{\left(\frac{1+\phi}{4}\right)x\right\} \in \left[0,\frac{1}{2}\right) \\
    2 : \left\{\left(\frac{1+\phi}{4}\right)x\right\} \in \left[\frac{1}{2},1\right)
    \end{cases}$$ avoids $4$-term monochromatic $F_E$-diffsequences.  \end{proof}
    
\end{proposition}

\section{Accessibility and Topological Dynamics} 

Many results in Ramsey theory can be reformulated in terms of topological dynamics. In fact, Bergelson and Liebman's original proof of the polynomial van der Waerden Theorem was topological \cite{bergelson1996polynomial}; only later on did Walters give a fully combinatorial proof of the result \cite{walters2000combinatorial}. In this section, we give a brief overview of some connections between accessibility and topological dynamics. For an introduction to topological dynamics, see \cite{furstenberg2014recurrence} (specifically Chapters $1,2,8$, and $9$). See also Robertson's paper \cite{robertson2020down} for a very nice overview of the connections between dynamics and largeness.

Going forward, we will use the following notation: Given a function $T: X \to X$, we let $T^n$ denote the $n$-th iterate of $T$, that is, 
$$
  T^n \;=\; \underbrace{T \circ T \circ \dotsb \circ T}_{n\text{ times}}.
$$ We also use the notation $T^nY=\{T^n(y) : y \in Y\}$ and $T^{-n}Y=\{x \in X : T^n(x) \in Y\}$ in place of $T^{n}(Y)$ and $T^{-n}(Y)$, respectively.

\begin{definition}

A \textit{topological dynamical system} is a pair $(X,T)$ consisting of a compact metric space $X$ and a homeomorphism $T : X \to X$. We say that $(X,T)$ is a \textit{minimal (topological) system} if no nontrivial closed subset of $X$ is $T$-invariant, that is, there is no closed $ \emptyset \not =K \subsetneq X$ such that $TK=K$.
    
\end{definition}

\begin{definition}

We say $D \subseteq \N$ is a \textit{set of $\ell$-topological recurrence} if for every minimal system $(X,T)$ and every (nonempty) open set $U \subseteq X$, there exists $d \in D$ such that $$U \cap T^{-d}U  \cap \dotsb \cap T^{-\ell d}U \not = \emptyset.$$ We say $D \subseteq \N$ is a \textit{set of topological recurrence} if $D$ is a set of $1$-topological recurrence, and we say $D$ is a \textit{set of multiple topological recurrence} if $D$ is a set of $\ell$-topological recurrence for every $\ell \geq 1$.
    
\end{definition}

It is well known that $D \subseteq \N$ is large if and only if $D$ is a set of multiple topological recurrence \cite{host2016variations}. It is also true that a set $D \subseteq \N$ is accessible if and only if $D$ is a set of topological recurrence. However, this latter fact has never been explicitly stated in the literature. For the sake of completeness, we provide a self-contained proof of this fact.

Before proceeding further, we make some general observations and introduce some notation. Firstly, we note that no generality is lost by working with colorings of $\Z$ instead of $\N$. This follows by a standard compactness argument and by identifying the interval $[-N,N]$ of $\Z$ with the interval $[1,2N+1]$ of $\N$. Alternatively: If every $r$-coloring of $\N$ admits arbitrarily long monochromatic $D$-diffsequences, then trivially every $r$-coloring of $\Z$ does as well. Conversely, assume every $r$-coloring of $\Z$ admits arbitrarily long monochromatic $D$-diffsequences. Given any $r$-coloring $\chi: \N \to [r]$ of $\N$, consider the $r$-coloring $\chi' : \Z \to [r]$ of $\Z$ defined by $\chi'(x)=\chi(|x|+1).$ Strictly negative or positive $k$-term monochromatic $D$-diffsequences under $\chi'$ correspond to $k$-term monochromatic $D$-diffsequences under $\chi$. Thus, any $2k$-term monochromatic $D$-diffsequence under $\chi'$ gives us a $k$-term monochromatic $D$-diffsequence under $\chi$. Since $k$ can be chosen arbitrarily, it follows that any $r$-coloring of $\N$ admits arbitrarily long monochromatic $D$-diffsequences.

Additionally, many of the results in this section are more concisely stated by considering an $r$-coloring $\chi : \Z \to [r]$ of $\Z$ as a partition $\Z=C_1 \cup C_2 \cup \dotsb \cup C_r$ of $\Z$ under the correspondence $n \in C_i$ if and only if $\chi(n)=i.$ Given a subset $D \subseteq \N$, we let $$\mathcal{DS}_k(D)=\left\{\{s_1,s_2,\dots,s_k\} \subseteq \Z : s_{i+1}-s_i \in D \text{ for all } 1 \leq i \leq k-1\right\},$$ that is, $\mathcal{DS}_k(D)$ is the set of all $k$-term $D$-diffsequences in $\Z$. Summarizing, saying $D \subseteq \N$ is $r$-accessible is equivalent to saying that for every partition $\Z=C_1 \cup C_2 \cup \dotsb \cup C_r$ of $\Z$ into $r$ subsets and every $k \geq 1$, there exists $i \in \{1,2,\dots,r\}$ and $S \in \mathcal{DS}_k(D)$ such that $S \subseteq C_i$.

We say that $\mathcal{A} \subseteq \mathcal{P}(\mathbb{Z})$ is \textit{translation invariant} if for every $A \in \mathcal{A}$ and every $n \in \mathbb{Z}$, we have $A+n \in \mathcal{A}$, where $A+n=\{a+n : a \in A\}.$ It is clear that $\mathcal{DS}_k(D)$ is translation invariant since if $S=\{s_1,s_2,\dots,s_k\}$ is a $k$-term $D$-diffsequence, then $S+n=\{s_1+n,s_2+n,\dots,s_k+n\}$ is as well. The following lemma is a topological correspondence principle of Furstenberg and Weiss \cite{furstenberg1978topological}. Essentially, given a partition $\Z=C_1 \cup C_2 \cup \dotsb \cup C_r$ of $\Z$ (i.e., an $r$-coloring of $\Z$), we can construct a minimal system $(X,T)$ and a clopen partition $\{U_i\}_{i=1}^{r}$ of $X$ that models the given partition of $\Z$.

\begin{lemma}
\label{lem:corr}

Let $\mathcal{A}$ be a translation invariant collection of finite subsets of $\Z$ and let $r \in \N$. The following are equivalent.

\begin{enumerate}[$(i)$]

    \item For any finite partition $\Z=C_1 \cup C_2 \cup \dotsb \cup C_r$ of $\Z$ into $r$ subsets, there exists $i \in \{1,2,\dots,r\}$ and $A \in \mathcal{A}$ such that $A \subseteq C_i.$

    \item Let $(X,T)$ be a topological dynamical system and let $\{U_i\}_{i=1}^{r}$ be an open cover of $X$. Then, there exists $i \in \{1,2,\dots,r\}$ and $A \in \mathcal{A}$ such that $$\bigcap_{a \in A}T^{-a}U_i \not = \emptyset.$$

    \item Let $(X,T)$ be a minimal system and let $\{U_i\}_{i=1}^{r}$ be a clopen cover of $X$. Then, there exists $i \in \{1,2,\dots,r\}$ and $A \in \mathcal{A}$ such that $$\bigcap_{a \in A}T^{-a}U_i \not = \emptyset.$$
    
\end{enumerate}

\begin{proof}

We begin by proving that $(i)$ implies $(ii)$. Let $(X,T)$ be a topological dynamical system and let $\{U\}_{i=1}^{r}$ be an open cover of $X$. Let $x \in X$ be arbitrary. Let $C_1=\{n \in \Z : T^nx \in U_1\}$ and $C_i=\{n \in \Z : T^nx \in U_i \setminus \bigcup_{k=1}^{i-1}U_k\}$ for $2 \leq i \leq r$. Then, $\Z=C_1 \cup C_2 \cup \dotsb \cup C_r$ is a partition of $\Z$, so there exists $i \in \{1,2,\dots,r\}$ and $A \in \mathcal{A}$ such that $A \subseteq C_i$. Thus, $T^ax \in U_i$ for all $a \in A$, and hence $$\bigcap_{a \in A}T^{-a}U_i \not = \emptyset.$$

Next, we note that $(ii)$ trivially implies $(iii)$ since by definition a minimal system is a topological dynamical system and a clopen cover of $X$ is an open cover of $X$.

Finally, we prove that $(iii)$ implies $(i)$. Let $\Z=C_1 \cup C_2 \cup \dotsb \cup C_r$ be a partition of $\Z$. Let $\Lambda=\{1,2,\dots,r\}$ and let $\Omega=\Lambda^\Z$ be the set of two-sided infinite words over the alphabet $\Lambda$. Define a metric on $\Omega$ by $d(x,y)=2^{-k}$, where $k={\min\{|n| : x_n \not = y_n\}}.$ This makes $\Omega$ a compact metric space. Let $T : \Omega \to \Omega$ denote the shift operator, which is defined by $Tx=y$, where $y_n=x_{n-1}$ for all $n \in \Z$. Then, $T$ is a homeomorphism (this is why we want to work with $\Z$ instead of $\N$). Let $\omega \in \Omega$ be defined by $\omega_n=i$ if and only if $n \in C_i$ and let $\overline{\mathcal{O}_T(\omega)}=\overline{\{T^n\omega : n \in \Z\}}$ denote the orbit closure of $\omega$. It follows that $\overline{\mathcal{O}_T(\omega)}$ is compact (since $\Omega$ is compact) and $T$-invariant. Thus, $(\overline{\mathcal{O}_T(\omega)},T)$ is a topological dynamical system. By considering the family $\mathcal{F}$ of nonempty closed $T$-invariant subsets of $\overline{\mathcal{O}_T(\omega)}$ ordered by inclusion, it follows by Zorn's Lemma that there exists a minimal subsystem $(X,T)$ of $(\overline{\mathcal{O}_T(\omega)},T)$ (see Chapter 1 of \cite{furstenberg2014recurrence} for details and Chapter 3 for a proof not relying on Zorn's Lemma). Consider the sets $U_i=\{x \in X : x_0=i\}$ for $1 \leq i \leq r$. The $U_i$'s are clopen and clearly $\{U_i\}_{i=1}^{r}$ covers $X$. Thus, there exists $i \in \{1,2,\dots,r\}$ and $A \in \mathcal{A}$ such that $\bigcap_{a \in A}T^{-a}U_i \not = \emptyset$. Thus, there exists $\upsilon \in X$ such that $\upsilon_{a}=i$ for all $a \in A$. Since $\upsilon \in \overline{\mathcal{O}_T(\omega)}$, there exists $m \in \Z$ such that $d(\upsilon,T^m\omega)<2^{-|\max{(A)}|}$. Thus, $\omega_{a+m}=\upsilon_a$ for all $a \in A$, so $A+m \subseteq C_i.$ Since $\mathcal{A}$ is translation invariant, $A+m \in \mathcal{A}$, and our proof is complete. \end{proof}
    
\end{lemma}

We are now in a position to prove that accessible sets are sets of topological recurrence and vice versa (the equivalence of $(i)$ and $(iv)$ is well known; see \cite{host2016variations} for this and more equivalencies).

\begin{theorem}
\label{theorem :recurrenceeq}
Let $D \subseteq \N$. The following are equivalent.

\begin{enumerate}[$(i)$]

\item $D$ is a set of topological recurrence.

\item Let $(X,T)$ be a minimal system and let $U \subseteq X$ be an open set. Then, for every $k \geq 1$, there exists a $k$-term $D$-diffsequence $\{s_1,s_2, \dots ,s_k\} \in \mathcal{DS}_k(D)$ such that $$U \cap T^{-s_1}U \cap \dotsb \cap T^{-s_k}U \not = \emptyset.$$

\item $D$ is accessible.

\item $D$ is chromatically intersective, that is, for any finite partition $\N=C_1 \cup C_2 \cup \dotsb \cup C_r$ of $\N$, there exists $i \in \{1,2,\dots,r\}$ and $a,b \in C_i$ such that $a-b \in D$.

\end{enumerate}

\end{theorem}

\begin{proof}

We begin by proving that $(i)$ implies $(ii).$ Let $(X,T)$ be a minimal system and let $U \subseteq X$ be an open set. Then, there exists $d_1 \in D$ such that $U \cap T^{-d_1}U  \not = \emptyset$. Since $U \subseteq X$ is open and $T : X \to X$ is a homeomorphism, $U \cap T^{-d_1}U$ is an open set of $X$. Thus, there exists $d_2 \in D$ such that $(U \cap T^{-d_1}U) \cap T^{-d_2}(U \cap T^{-d_1}U) \not = \emptyset$. In particular, $U \cap T^{-d_1}U \cap T^{-d_1-d_2}U$ is an open set of $X$. Iterating this process gives, for any $k \geq 1$, a $k$-term $D$-diffsequence $\{s_1,s_2,\dots,s_k\} \in \mathcal{DS}_k(D)$ such that $$U \cap T^{-s_1}U \cap \dotsb \cap T^{-s_k}U \not = \emptyset,$$  where $s_j=\sum_{i=1}^{j}d_i$. 

Now, we prove that $(ii)$ implies $(iii)$. Recall that $\mathcal{DS}_k(D)$ is a translation invariant collection of finite subsets of $\Z$, that is, the set of all $k$-term $D$-diffsequences in $\Z$ is translation invariant. Assume, for the sake of contradiction, that $D$ is not $r$-accessible for some $r \in \N$. Then, there exists a partition $\Z=C_1 \cup C_2 \cup \dotsb \cup C_r$ of $\Z$ and a positive integer $k$ such that for every $1 \leq i \leq r$ and every $S\in \mathcal{DS}_k(D)$, we have $S \not \subseteq C_i$. Thus, by Lemma \ref{lem:corr}, there exists a minimal system $(X,T)$ and a clopen cover $\{U_i\}_{i=1}^{r}$ of $X$ such that for every $1 \leq i \leq r$ and every $S \in \mathcal{DS}_k(D)$, we have $\bigcap_{s \in S}T^{-s}U_i=\emptyset.$ In other words, there exists an open set $U_i$ of $X$ such that for every $k$-term $D$-diffsequence $\{s_1,s_2,\dots,s_k\} \in \mathcal{DS}_k(D),$ $$U_i \cap T^{-s_1}U_i \cap \dotsb \cap T^{-s_k}U_i = \emptyset,$$ contradicting $(ii)$.

The fact that $(iii)$ implies $(iv)$ follows immediately from the definition of accessibility.

Finally, we prove that $(iv)$ implies $(i)$. Let $D$ be chromatically intersective. Assume, for the sake of contradiction, that $D \subseteq \N$ is not a set of topological recurrence. Then, there exists a minimal system $(X,T)$ and an open set $U \subseteq X$ such that for every $d \in D$, $U \cap T^{-d}U = \emptyset$. Since $(X,T)$ is a minimal system, $\bigcup_{n=0}^{\infty}T^{-n}U=X$. Since $T$ is a homeomorphism and $X$ is compact, there exists $r \in \N$ such that $\bigcup_{n=0}^{r-1}T^{-n}U=X$. Let $x \in X$ be arbitrary. Let $C_1=\{n \in \N: T^nx \in U\}$ and let $C_i=\{n \in \N : T^nx \in T^{-i+1}U\} \setminus \bigcup_{k=1}^{i-1}C_k$ for $2 \leq i \leq r$. It follows that $\N=C_1 \cup C_2  \cup \dotsb \cup C_{r}$ is a partition of $\N$. Since we assumed $D$ is not a set of topological recurrence, there exists no $i \in \{1,2,\dots,r\}$ and $a,b \in C_i$ such that $a-b \in D$, and thus $D$ is not chromatically intersective, a contradiction. This completes our proof. \end{proof}

With the equivalence between accessibility and topological recurrence explicitly stated there are now many interesting examples of sets that are accessible (sets of topological recurrence), as well as sets that are accessible but not large (sets of topological recurrence but not multiple topological recurrence). Indeed, Furstenberg constructed a set of topological recurrence that is not a set of $2$-topological recurrence in \cite{furstenberg2014recurrence} (Chapter 9, Pages 177-178). This result was later reformulated by Jungi\'c to answer a conjecture of Brown asking if every accessible set is large \cite{jungic2005conjecture}. In \cite{frantzikinakis2006sets}, the authors prove that for any $k \geq 2$ and $\alpha \in \R$, the set $D=\{n \in \N : \{\alpha n^k \} \in [1/4,3/4]$\} is a set of $(k-1)$-topological recurrence but not a set of $k$-topological recurrence. Many more interesting examples of sets of topological recurrence (and thus accessible sets) can be found in \cite{bourgain1987ruzsa,frantzikinakis2006sets,furstenberg2014recurrence,griesmer2019recurrence,griesmer2021separating,griesmer2024set,mccutcheon1995three}.

\begin{remark}

We say that $D \subseteq \N$ is a set of \textit{$\ell$-measurable recurrence} if for every measure preserving system $(X,\mathscr{B},\mu,T)$ and any measurable set $A \in \mathscr{B}$ with $\mu(A)>0$, there exists $d \in D$ such that $$\mu(A \cap T^{-d}A \cap \dotsb \cap T^{-\ell d}A)>0.$$ Similarly to the case of topological recurrence, we say that $D$ is a set of \textit{measurable recurrence} if $D$ is a set of $1$-measurable recurrence and a set of \textit{multiple measurable recurrence} if $D$ is a set of $\ell$-measurable recurrence for every $\ell \geq 1$. It is well known that any set of $\ell$-measurable recurrence is a set of $\ell$-topological recurrence. However, the converse need not be true, as K{\v{r}}{\'\i}{\v{z}} constructed a set of topological recurrence that is not a set of measurable recurrence in \cite{kvrivz1987large} (an exposition of this construction is given in \cite{mccutcheon1995three}). We mention this since it follows that sets of multiple measurable recurrence are large, and hence interesting examples of large sets can be found in \cite{bergelson2000ergodic} (Theorems 0.9 and 0.10), \cite{bergelson2008intersective} (Theorems 1.3 and 1.4), and \cite{tsinas2023joint} (Corollaries 1.5 and 1.6).
    
\end{remark}

Of course, any property possessed by sets of topological recurrence must also be a property of accessible sets. Consider, for example, the following simple property possessed by sets of $k$-topological recurrence (a simple proof is given in \cite{host2016variations}, Proposition 6.2).

\begin{theorem}

If $D=D_1 \cup D_2$ is a set of $k$-topological recurrence, then either $D_1$ or $D_2$ is a set of $k$-topological recurrence.
    
\end{theorem}

This simple fact about sets of $k$-topological recurrence implies both Theorem \ref{theorem :largeunion} and the following corollary.

\begin{corollary}
\label{cor:3}
If $D=D_1 \cup D_2$ is accessible, then either $D_1$ or $D_2$ is accessible.
    
\end{corollary}

\begin{remark}

Corollary \ref{cor:3} is easy to prove if ``accessible'' is replaced by ``large'': If $D_1$ is not $r$-large and $D_2$ is not $s$-large, then given colorings $\chi_1 : \N \to [r]$ and $\chi_2 : \N \to [s]$ that avoid arbitrarily long monochromatic $D_1$-APs and $D_2$-APs, respectively, their standard product coloring $\chi : \N \to [r] \times [s]$ defined by $\chi(n)=(\chi_1(n),\chi_2(n))$ avoids arbitrarily long monochromatic $D_1 \cup D_2$-APs. However, this product coloring does not necessarily work when considering $D_1 \cup D_2$-diffsequences unless we have the added assumption that $D_1+D_2 \subseteq D_1$ or $D_1 + D_2 \subseteq D_2$ (\cite{landman2014ramsey}, Lemma 10.24).
    
\end{remark}

In \cite{host2016variations}, the authors give a dynamical formulation for $r$-largeness. Namely, $D \subseteq \N$ is \text{$r$-large} if for every topological dynamical system $(X,T)$, every open cover $\{U_i\}_{i=1}^{r}$ of $X$ with $r$ open sets, and every $k \geq 1$, there exists $i \in \{1,2,\dots,r\}$ and $d \in D$ such that $$U_i \cap T^{-d}U_i  \cap \dotsb \cap T^{-kd}U_i \not =\emptyset.$$ In a similar fashion, we can give a dynamical formulation for $r$-accessibility.

\begin{definition}[$r$-accessibility, topological version]
\label{definition :racctop}

A set $D \subseteq \N$ is \textit{$r$-accessible} if for every minimal system $(X,T)$, every clopen cover $\{U_i\}_{i=1}^{r}$ of $X$ with $r$ clopen sets, and every $k \geq 1$, there exists $i \in \{1,2,\dots,r\}$ and a $k$-term $D$-diffsequence $\{s_1,s_2, \dots ,s_k\} \in \mathcal{DS}_k(D)$ such that $$U_i \cap T^{-s_1}U_i \cap \dotsb \cap T^{-s_k}U_i \not = \emptyset.$$ 
    
\end{definition}

\begin{remark}

That this coincides with our original definition of $r$-accessibility follows from Lemma \ref{lem:corr} (let $\mathcal{A}=\mathcal{DS}_k(D)$ for $k \geq 1$). In addition, as seen in the proof of Lemma \ref{lem:corr}, one may assume the sets $\{U_i\}_{i=1}^{r}$ in the cover  of $X$ are pairwise disjoint. Similarly, in the topological formulation of $r$-largeness, one need only consider minimal systems $(X,T)$ that possess a clopen cover $\{U_i\}_{i=1}^{r}$ of $X$ into $r$ pairwise disjoint clopen sets.

Recall the $2$-Large Conjecture, which asserts that if $D \subseteq \N$ is $2$-large, then $D$ is large \cite{brown1999set}. This conjecture can be restated in the language of topological dynamics.

\begin{conjecture}[2-Large Conjecture, topological version]

Let $D \subseteq \N$. Assume that for every minimal system $(X,T)$, every partition $X=U_1 \cup U_2$ of $X$ into $2$ clopen sets, and every $k \geq 1$, there exists $i \in \{1,2\}$ and $d \in D$ such that $$U_i \cap T^{-d} U_i \cap \dotsb \cap T^{-kd}U_i \not = \emptyset.$$ Then, for every minimal system $(X,T)$, every open set $U \subseteq X$, and every $k \geq 1$, there exists $d \in D$ such that $$U \cap T^{-d} U \cap \dotsb \cap T^{-kd}U \not = \emptyset.$$ 
    
\end{conjecture}  
    
\end{remark}

We now expand some on the family of colorings used in Theorem \ref{theorem :raccgrowth}, Proposition \ref{cor:fibAP}, and Proposition \ref{cor:fibdiff}. Let $\mathbb{T}=\R/\Z$ be the torus (real numbers modulo $1$). $\mathbb{T}$ can be equipped with the metric $d(x,y)=\min(|x-y|,1-|x-y|)$ for all $x,y \in \mathbb{T}$, or equivalently $d(x,y)=\min(\{x-y\},1-\{x-y\})$ for all $x,y \in \R$. This turns $\mathbb{T}$ into a compact metric space. Let $C_1 \cup C_2 = \mathbb{T}$ be a partition of the torus such that there exists an open set $U \subseteq \mathbb{T}$ with either $U \subseteq C_1$ or $U \subseteq C_2$. Let $\alpha \in \mathbb{T}$. Let $T_{\alpha}: \mathbb{T \to \mathbb{T}}$ be defined by $T_\alpha(x)=x+\alpha$. Let $x_0 \in X$. Define a $2$-coloring $\chi : \mathbb{N} \to \{1,2\}$ of $\N$ by $$\chi(n)=\begin{cases}
    1 : T_\alpha^{n}(x_0) \in C_1 \\
    2: T_\alpha^n(x_0) \in C_2.
\end{cases}$$ This is just a generalization of the colorings given in the remark following Proposition \ref{cor:fibdiff}, where we had $C_1=\left[0,\frac{1}{2}\right)$, $C_2=\left[\frac{1}{2},1\right)$, and $x_0=0$. 

\begin{definition}

A set $D \subseteq \N$ is called $\operatorname{Bohr}(n)$ if for every $\alpha_1,\alpha_2,\dots,\alpha_n \in \mathbb{T}$ and every $\varepsilon>0$, there exists $d \in D$ such that $\|\alpha_1 d\| < \varepsilon,\|\alpha_2 d\| < \varepsilon,\dots,\|\alpha_n d\| <\varepsilon$.

\end{definition}

An identical argument to that of Theorem 8.10 in \cite{farhangi2016refinements} shows $D \subseteq \N$ is $\operatorname{Bohr}(1)$ if and only if every $2$-coloring defined as above admits arbitrarily long monochromatic $D$-APs. It is known that $\operatorname{Bohr}(1)$ sets are not necessarily $2$-large \cite{farhangi2016refinements}; however, we wonder if they are $2$-accessible, that is, is it true that $$\{D \subseteq \N : D \text{ is } 2\text{-large}\} \subseteq \{D \subseteq \N : D \text{ is } \operatorname{Bohr}(1)\} \subseteq \{D \subseteq \N : D \text{ is } 2\text{-accessible}\}?$$

A very interesting open problem is that of determining if the primes are $2$-accessible (a simple $3$-coloring shows that they are not $3$-accessible \cite{landman2007avoiding}). Consider the topological dynamical system $(\mathbb{T},T_\alpha)$. Let $P$ be the set of prime numbers. We claim that for every open cover $\{U_1,U_2\}$ of $\mathbb{T}$ and every $k \geq 1$, there exists $i \in \{1,2\}$ and a $k$-term $P$-diffsequence $\{s_1,s_2,\dots,s_k\} \in \mathcal{DS}_k(P)$ such that $$U_i \cap T_{\alpha}^{-s_1}U_i \cap \dotsb \cap T_{\alpha}^{-s_k}U_i \not = \emptyset.$$ Indeed, if $\alpha$ is irrational, then our claim follows immediately from the fact that for any $x \in \mathbb{T}$, the set $T_\alpha^{p}(x)=\{p \alpha +x\}$ is dense in $\mathbb{T}$ as $p$ ranges through the primes \cite{vaughan1977distribution}. Now, if $\alpha=a/b$ is a rational in lowest terms, then $\mathcal{O}_{T_{\alpha}}(0)=\{0,T_{\alpha}(0),\dots,T_{\alpha}^{b-1}(0)\}$. Suppose that $T_{\alpha}^{k}(0),T_{\alpha}^{k+1}(0) \in U_i$ for some $i \in \{1,2\}$. By Dirichlet's Theorem \cite{apostol1998introduction}, for any $b>1$ there exist infinitely many primes $p \equiv -1 \pmod{b}$ and $p \equiv 1 \pmod{b}$. Thus, since $T_{\alpha}^{k+p}(0)=T_{\alpha}^{k+1}(0)$ if $p \equiv 1 \pmod{b}$ and $T_{\alpha}^{k+1+p}(0)=T_{\alpha}^{k}(0)$ if $p \equiv -1 \pmod{b}$, the result follows. Otherwise, without loss of generality, $0 \in U_1$, $T_{\alpha}(0) \in U_2$, $T_{\alpha}^{2}(0) \in U_1, T_{\alpha}^{3}(0) \in U_2$, and so on. Thus, the result follows since $2$ is a prime (this is just a slightly different formulation of the remark at the end of \cite{landman2010avoiding}). We wonder if this is enough to deduce the primes are $2$-accessible.

\begin{question}

Let $D \subseteq \N$. Assume that for every topological dynamical system $(\mathbb{T},T_\alpha)$, every open cover $\{U_1,U_2\}$ of $\mathbb{T}$, and every $k \geq 1$, there exists $i \in \{1,2\}$ and a $k$-term $D$-diffsequence $\{s_1,s_2,\dots, s_k\} \in \mathcal{DS}_k(D)$ such that $$U_i \cap T_{\alpha}^{-s_1}U_i \cap \dotsb \cap T_{\alpha}^{-s_k}U_i \not = \emptyset.$$ Is $D$ $2$-accessible? 
    
\end{question}

To expand a little more on Question 1, we are essentially asking if it is enough to check only periodic $2$-colorings and $2$-colorings generated by irrational rotations of $\mathbb{T}$ to deduce $2$-accessibility. Informally, given a $2$-coloring $\chi: \N \to \{1,2\}$ of $\N$, we can think of it as an infinite word $\chi \in \{1,2\}^\N$, where $\chi_n=\chi(n)$, and define its complexity function $p(\chi,n)$ to be the number of unique length $n$ strings appearing in $\chi$. Clearly $p(\chi,n) \leq 2^n$, and it is well known that if $p(\chi,n) \leq n$ for any $n \in \N$, then $\chi$ is eventually periodic (\cite{lothaire2002algebraic}, Theorem 1.3.13). Colorings defined by irrational circle rotations (as in Theorem \ref{theorem :raccgrowth}) have low complexity while still being aperiodic. In fact, Sturmian words are words $\omega \in \{1,2\}^\N$ with complexity $p(\omega,n)=n+1$ for all $n \in \N$, and every Sturmian word can be generated by an irrational rotation (\cite{lothaire2002algebraic}, Theorem 2.1.13). It follows from Borel's Normal Number Theorem that almost all (with respect to Lebesgue measure) $r$-colorings $\chi \in \{1,2,\dots,r\}^\N$ have complexity $p(\chi,n)=r^n$ for all $n \in \N$. It seems natural to want to look at $r$-colorings with low complexity when attempting to prove a set is not $r$-accessible. We think the following question is interesting and further motivates Question 1.

\begin{question}

Let $D \subseteq \N$ satisfy $\operatorname{doa}(D) <r.$ Does there exist an $r$-coloring $\chi \in \{1,2,\dots,r\}^\N$ of $\N$ that does not admit arbitrarily long monochromatic $D$-diffsequences and which satisfies $p(\chi,n)=O(n)$? 
    
\end{question}

As seen in the remark following Theorem \ref{theorem :raccgrowth}, there is a connection between the chromatic number, $\chi(\N_D)$, of the distance graph $\N_D$ and the degree of accessibility, $\operatorname{doa}(D)$, of $D$. Indeed, for any $D \subseteq \N$, there exists $C>0$ such that $$\operatorname{doa}(D) < \chi(\N_D) \leq C \operatorname{doa}(D).$$ We wonder what more can be said about this relationship and what changes if $\mathbb{N}$ is replaced with $\mathbb{Z}$ or $\mathbb{R}$. Letting $V_m=\{n \in \N : m \nmid n\}$, we have, as shown in \cite{landman2014ramsey}, that $\operatorname{doa}(V_m)=m-1$. In addition, the coloring used to show $\operatorname{doa}(V_m)<m$ avoids $2$-term monochromatic $V_m$-diffsequences, so $\chi(\N_{V_m})=m$. Thus, $$\frac{\chi(\N_{V_m})}{\operatorname{doa}(V_m)}=\frac{m}{m-1} \to 1 \quad \text{as} \quad m \to \infty.$$ On the other hand, let $D=\{1,2,\dots,m-1\} \cup \{p^n : n \in \N\}$, where $m \geq 3$ and $p>m$ is any prime not dividing $m$. Since $\N_D$ contains a complete subgraph $K_m$, we have $\chi(\N_D) \geq m$. In addition, the $m$-coloring $\chi: \N \to \{0,1,\dots,m-1\}$ defined by $\chi(n)=i$ if and only if $n \equiv i \pmod{m}$ gives us $\chi(\N_D) \leq m$, so $\chi(\N_D)=m$. Finally, by Theorem \ref{theorem :raccgrowth}, $\operatorname{doa}(D)=1$, so $$\frac{\chi(\N_D)}{\operatorname{doa}(D)}=m.$$

\begin{question}

What properties does the set $$X=\left\{\frac{\chi(\N_D)}{\operatorname{doa}(D)} : D \subseteq \N\right\}$$ have? The above two examples show that $\N \cap (1,\infty) \subseteq X$. Is $X=\mathbb{Q} \cap (1,\infty)?$ 
    
\end{question}

\begin{remark}

As shown in \cite{katznelson2001chromatic}, $\chi(\Z_D)=\infty$ if and only if $D$ is a set of topological recurrence. Accepting this, it is easy to see that $D$ is accessible if and only if $D$ is a set of topological recurrence. Indeed, if $D$ is accessible, it is clear that $\chi(\mathbb{Z}_D)=\infty$. Conversely, if $\operatorname{doa}(D)$ is finite, then there exists a finite coloring $\chi : \mathbb{Z} \to [r]$ having all monochromatic $D$-diffsequences of bounded length. Thus, $\chi' : \mathbb{Z} \to [r] \times [\ell]$ defined by $\chi'(n)=(\chi(n),\ell(n))$, where $\ell(n)=$ length of longest monochromatic $D$-diffsequence starting at $n$ and $\ell=\max\{\ell(n) : n \in \Z\}$, is a proper coloring of $\chi(\Z_D)$.
    
\end{remark}

\section{Some Concluding Remarks}

We are not sure whether the condition $d_{n+1} \geq (3+\delta)d_n$ in Corollary \ref{cor:2accgrowth} is the best possible. We hoped to weaken the condition to $d_{n+1} \geq (2+\delta)d_n$, but were unsuccessful. There is of course an added difficulty when we wish to have $\{\alpha d_n\} \in [\varepsilon,1/2]$ as opposed to $\|\alpha d_n\| \geq \varepsilon$. In the latter case, if we have an interval with a small number of `bad' fractions of the form $a/d_n$, then we can remove small $\varepsilon$-neighborhoods centered at these fractions and be left with a union of closed intervals avoiding these bad fractions. If the sequence is lacunary, this idea can be used in a clever way to construct a sequence of nested intervals that converge to some $\alpha \in \R$ with the property that $\|\alpha d_n\| \geq \varepsilon$ for all $n \in \N$. However, if we want $\{\alpha d_n\} \in [\varepsilon,1/2]$, then given any interval, it is not enough to simply remove $\varepsilon$-neighborhoods; given a bad fraction $a/d_n$ in the interval, we need to remove $\left(\frac{a-1/2}{d_n},\frac{a+\varepsilon}{d_n}\right)$, that is, we need to remove intervals of size $\frac{1/2+\varepsilon}{d_n}$ as opposed to intervals of size $\frac{2\varepsilon}{d_n}$. We think this is where most of the added difficulty lies. A somewhat related problem to this (\cite{dubickas2006even}, Problem $5$) asks whether or not there exists $\beta > 2$ such that $\{\alpha  \beta^n\} \geq 1/2$ infinitely often for all $\alpha \not = 0$.

Perhaps there exists some $2$-accessible set $D=\{d_1,d_2,d_3,\dots\}$ such that $d_{n+1} \geq (2+\delta)d_n$ for all $n \geq 1$? For example, it is not known whether the set of Pell numbers, which satisfy the recurrence $p_n=2p_{n-1}+p_{n-2}$, is $2$-accessible or not.

We would also like to determine if the set of Fibonacci numbers is $3$-accessible. We cannot apply Theorem \ref{theorem :raccgrowth} to conclude the set is not $3$-accessible since for any $\varepsilon>0$ and $\alpha \in \R$, it is impossible to have $\{\alpha f_n\} \in \left[\varepsilon,\frac{2}{3}\right]$ for all $n \in \N$. See \cite{dubickas2009approximation}, which proves that for every $\alpha \not \in \mathbb{Q}(\sqrt{5})$, there exists infinitely many $n \in \N$ such that $\{\alpha f_n \} > 2/3$. If $\alpha \in \mathbb{Q}(\sqrt{5})$, then (using the fact that $\sqrt{5}f_n=f_{n-1}+f_{n+1}-2\bar{\phi}^n$) a fairly straightforward argument can show that for any $\varepsilon>0$, it is impossible to have $\{\alpha f_n\} \in \left[\varepsilon,\frac{2}{3}\right]$ for all $n \in \N$.

Finally, it would be interesting to find a set $D$ that is $2$-accessible (and not $2$-large) that does not contain arbitrarily long monochromatic $D$-diffsequences. We have yet to see a proof of a set being $2$-accessible that does not use Corollary \ref{theorem :2acccon} or the fact that $D-D=\{d_2-d_1 : d_1,d_2 \in D, d_2-d_1 >0\}$ is accessible for any infinite set $D$. Perhaps the dynamical formulation of $2$-accessibility could help shed light on some new $2$-accessible sets?

\vskip20pt\noindent {\bf Acknowledgments.} The author extends his heartfelt gratitude to Dr.~Jackie Anderson, whose generous mentorship, insight, and encouragement shaped this work -- originally developed as part of an undergraduate thesis -- and who continues to be a profound source of inspiration. The author also wishes to thank an anonymous referee for highlighting the connections between accessibility and topological recurrence and for their many other valuable suggestions. Part of this work was supported during the summer 
of 2023 through Bridgewater State University's ATP research grant.

\end{document}